\documentclass[a4paper]{article}
\usepackage{amsmath,amssymb,amscd,amsthm}
\title{Representation of solutions of the Gauss hypergeometric equation by the multiple polylogarithms, functional relations of the multiple polylogarithms and relations of the multiple zeta values}
\author{Shu Oi}
\DeclareFontEncoding{OT2}{}{}
\DeclareFontSubstitution{OT2}{wncyr}{m}{n}
\DeclareSymbolFont{cyss}{OT2}{wncyss}{m}{n}
\DeclareSymbolFont{cyr}{OT2}{wncyr}{m}{n}
\DeclareMathSymbol{\sh}{\mathbin}{cyss}{`x}
\newcommand{\Li}{\operatorname{Li}}
\newcommand{\vect}[2]{{\begin{pmatrix}{#1}\\{#2}\end{pmatrix}}}
\newcommand{\reg}{\operatorname{reg}}
\newcommand{\dist}{\operatorname{dist}}
\newcommand{\length}{\operatorname{length}}

\newcommand{\C}{{\mathbf C}}
\newcommand{\R}{{\mathbf R}}

\newcommand{\Z}{{\mathbf Z}}

\newcommand{\bP}{{\mathbf P}}

\newcommand{\bzeta}{\text{\boldmath $\zeta$}}

\newcommand{\fH}{{\mathfrak H}}
\newcommand{\fh}{{\mathfrak h}}
\newcommand{\fU}{{\mathfrak U}}
\newcommand{\fI}{{\mathfrak I}}

\newcommand{\ds}{\displaystyle}

\newtheorem{thm}{Theorem}
\newtheorem{cor}[thm]{Corollary}
\newtheorem{prop}[thm]{Proposition}
\newtheorem{lem}[thm]{Lemma}

\begin{document}

\allowdisplaybreaks

\maketitle

\begin{abstract}
In this article, we express solutions of the Gauss hypergeometric equation as a series of the multiple polylogarithms by using iterated integral. This representation is the most simple case of a semisimple representation of solutions of the formal KZ equation. Moreover, combining this representation with the connection relations of solutions of the Gauss hypergeometric equation, we obtain various relations of the multiple polylogarithms of one variable and the multiple zeta values.
\end{abstract}

\section{Introduction} \label{sec:Introduction}

In this decade, algebraic theory of the formal Knizhnik-Zamolodchikov (KZ, for short) equation and the multiple polylogarithms has been explored with development of study on the multiple zeta values \cite{D}\cite{Go}\cite{MPH}\cite{OU}. However, application of the theory to specific differential equations does not seem to be studied so far. In this article, we consider the Gauss hypergeometric equation, which is thought to be the most fundamental example of a semisimple representation of the formal KZ equation, and obtain an analytic iterated integral expression of solutions to the Gauss hypergeometric equation. As a result, we can show that the generating function of the multiple polylogarithms of one variable of fixed weight, depth and height, which firstly appeares in the work of Ohno-Zagier \cite{OZ}, is naturally got as iterated integral solutions.

Solutions to the Gauss hypergeometric equation have integral expression of the Euler type or the Barnes type and thereby the connection formulas of the solutions are completely determined. Combining these connection formulas with the above results, we have various functional equations of the multiple polylogarithms of one variable. The category of the functional equations obtained like this involves Euler's inversion formula for polylogarithms and other interested examples.

\paragraph{}

We achieve two main purposes in this article. Firstly, we express solutions to the hypergeometric equation as a series of the multiple polylogarithms of one variable. This series involves the parameters of the hypergeometric equation. For the purpose, we regard the hypergeometric equation as a representation of the formal KZ equation and compute a representation of the fundamental solution to the formal KZ equation which satisfies certain asymptotic behavior at $z=0,1$. Furthermore we show that this series is absolutely convergent on any compact subset of the universal covering space of $\bP^1-\{0,1,\infty\}$.

Secondly, applying to this representation to the connection formula of solutions to the hypergeometric equation, we obtain functional relations of the multiple polylogarithms of one variable. Making specialization of $z=1$ therein, we get diverse relations of the multiple zeta values.

\paragraph{}

The contents of this article is the first step of an investigation of the formal KZ equation and its representations. In the future, we will consider higher dimensional representations and representations of generalized formal KZ equations of many-variables and derive diverse relations of multiple polylogarithms of many variables, in this context.

\paragraph{}

This article is organized as follows: In the remains of section 1, we preliminary recall the fundamental properties and known results on the shuffle algebra, the multiple polylogarithms, the formal KZ equation and the Gauss hypergeometric equation. In section 2, we consider the analytic continuation of the multiple polylogarithms of one variable to the universal covering space of $\bP^1-\{0,1,\infty\}$ as an analytic function.

In section 3, we give the representation of the solution to the Gauss hypergeometric equation which is regular at $z=0$ by the multiple polylogarithms of one variable. This is one of the main results in this article.

In section 4, we derive various functional relations among the multiple polylogarithms of one variable by virtue of comparing the result of section 3 and the regular part of the connection formulas of the solutions of the Gauss hypergeometric equation. In section 5, we consider irregular parts of the connection formulas and obtain some interesting functional relations of the multiple polylogarithms of one variable and relations of the multiple zeta values. These relations are also the main results in this article.

\subsection{The shuffle algebra and the multiple polylogarithms of one variable} \label{subsec:Intro:MPL}

The multiple polylogarithms of one variable are many-valued analytic functions on $\bP^1-\{0,1,\infty\}$ and play the essential role in study on the multiple zeta values or the formal KZ equation. We briefly review their algebraic aspects according to \cite{OU}.

\paragraph{The shuffle algebra}

Let $\fh=\C\langle x,y\rangle$ be a non-commutative polynomial algebra in letters $x,y$ over $\C$ and let $\fh^1=\C+\fh y$ and $\fh^0=\C+x \fh y$ be the subalgebras of $\fh$.

Following \cite{R}, we introduce the shuffle product $\sh$ on $\fh$ by an inductive way such as
\begin{align}
&1 \sh w = w \sh 1 = w, \\
&a_1 w_1 \sh a_2 w_2 = a_1(w_1 \sh a_2w_2) + a_2(a_1 w_1 \sh w_2),
\end{align}
where the notations $a_1$ and $a_2$ stand for the letter $x$ or $y$, and the notations $w_1$ and $w_2$ words of $\fh$. Then the shuffle algebra $(\fh, \sh)$ is a commutative and associative algebra over $\C$ (\cite{R}), and subspaces $\fh^1$ and $\fh^0$ are shuffle subalgebras of $\fh$. We note that $\fh$ is regarded as a polynomial algebra of $x$ over $\fh^1$, and furthermore, a polynomial algebra of $x,y$ over $\fh^0$. Namely,

\begin{align}
\fh &= \bigoplus_{n=0}^{\infty}\fh^1 \sh x^{\sh n} (= \fh^1[x]) \label{regdef}\\
&= \bigoplus_{m,n=0}^{\infty}\fh^0 \sh x^{\sh m} \sh y^{\sh n} (= \fh^0[x,y]) \label{regdef0}.
\end{align}
\noindent
By virtue of this isomorphism, one can define the regularization map $\reg^i(i=0,1): \fh \to \fh^i$ as follows;
\begin{align}
\reg^1(w) &= \text{the constant term of $w$ in the decomposition \eqref{regdef}}, \\
\reg^0(w) &= \text{the constant term of $w$ in the decomposition \eqref{regdef0}}.
\end{align}
The regularization map $\reg^1$ satisfies the following properties;
\begin{align}
w x^n &= \sum_{j=0}^n \reg^1(w x^{n-j}) \sh x^j \quad \text{for } w \in \fh^{1}, \label{eq:reg1}\\
\reg^1(w yx^n)&=(-1)^n(w \sh x^n)y \quad \text{for } n \ge 0, w \in \fh. \label{eq:reg2}
\end{align}

\paragraph{The multiple polylogarithms of one variable}

For the word \\$w = x^{k_1-1}y \cdots x^{k_r-1}y$ in $\fh^1$, the multiple polylogarithm of one variable (MPL for short) $\Li(w;z)$ is defined by 
\begin{align}
\Li(w;z) &= \Li_{k_1,\ldots,k_r}(z) := \sum_{m_1>m_2>\cdots>m_n>0}\frac{z^{m_1}}{m_1^{k_1}\cdots m_n^{k_n}},\label{MPLdef} \\
\Li(1;z)&:=1. \notag
\end{align}
The power series in the right hand side is absolutely convergent for $|z|<1$. If the word $w = x^{k_1-1}y \cdots x^{k_r-1}y$ belongs to $\fh^0$ (namely $k_1>1$), the corresponding MPL converges at $z=1$ and gives the multiple zeta value (MZV for short);
\begin{equation}
\zeta(w)=\zeta(k_1,\ldots,k_r) := \lim_{z \to 1-0}\Li(w;z) = \sum_{m_1>m_2>\cdots>m_n>0}\frac{1}{m_1^{k_1}\cdots m_n^{k_n}}.
\end{equation}

We extend MPL to $\fh$ by putting for $w \in \fh^1$,
\begin{equation}
\Li(wx^n;z) = \sum_{j=0}^n \Li(\reg^1(w x^{n-j});z)\frac{\log^j z}{j!}.
\end{equation}
Then $\Li(\bullet;z)$ is a $\sh$-homomorphism from $\fh$ to $\C$, namely one has $\Li(w \sh w';z)=\Li(w;z)\Li(w';z)$ for $w,w' \in \fh$. Furthermore the extended MPL satisfies the following differential recursive relations(\cite{O});
\begin{align}
\frac{d \Li(xw;z)}{dz} &= \frac{1}{z}\Li(w;z), \label{MPL_diff_x}\\
\frac{d \Li(yw;z)}{dz} &= \frac{1}{1-z}\Li(w;z). \label{MPL_diff_y}
\end{align}

By virtue of these differential relation, one can check that MPL has an iterated integral representation as follows;
\begin{multline}
\Li(x^{k_1-1}y \cdots x^{k_r-1}y;z) = \int_0^z \underbrace{\frac{dt}{t} \circ \cdots \circ \frac{dt}{t}}_{k_1-1 \text{ times}} \circ \frac{dt}{1-t} \circ \underbrace{\frac{dt}{t} \circ \cdots \circ \frac{dt}{t}}_{k_2-1 \text{ times}} \circ \frac{dt}{1-t} \\
\circ \cdots \circ \underbrace{\frac{dt}{t} \circ \cdots \circ \frac{dt}{t}}_{k_r-1 \text{ times}} \circ \frac{dt}{1-t}. \label{rep_ite_MPL}
\end{multline}
\noindent where $\ds \int_0^z \omega_1(t) \circ \omega_2(t) \circ \cdots \circ \omega_r(t)$ (each $\omega_i(t)$ is a $1-form$ of $t$) stands for an iterated integral $\ds \int_0^z \omega_1(t_1) \int_0^{t_1} \omega_2(t_2) \cdots \int_0^{t_{r-1}}\omega_r(t_r)$. The representation yields an analytic continuation of MPL on $\bP^1-\{0,1,\infty\}$ as a many-valued analytic function.

\paragraph{The weight, depth, height of words}

For any word $w$ in $\fh$, we define the weight $|w|$, the depth $d(w)$ and the height $h(w)$ of $w$ by the following;
\begin{align}
|w| &:= \text{the number of letters in } w,\\
d(w) &:= \text{the number of $y$ which appears in $w$},\\
h(w) &:= \text{(the number of $yx$ which appears in $w$)} + 1.
\end{align}

Denote by $g_i(k,n,s) \; (i=0,1)$ the sum of all words in $\fh^i$ with fixed weight $k$, depth $n$ and height $s$; namely
\begin{equation}
g_i(k,n,s) = \sum_{\substack{w \in \fh^i\\|w|=k,\; d(w)=n\\ h(w)=s}} w. \label{def_g0}\\
\end{equation}
Set $g_0(k,n,s;z)=0$ if $k<n+s$, $n<s$ or $k,n,s \in \Z_{\le 0}$, and $g_1(k,n,s;z)=0$ if $k<n+s-1$, $n<s$ or $k,n,s \in \Z_{\le 0}$. We note that, if the word $w=x^{k_1-1}y\cdots x^{k_r-1}y$ belongs to $\fh^0$, $h(w)=\#\{i|k_i \ge 2\}$. This is the original definition of the height given in \cite{OZ}. Hence one obtains the expression
\begin{equation}
g_0(k,n,s) = \sum_{\substack{k_1+\cdots+k_n=k\\k_1 \ge 2,\; k_2,\ldots,k_n \ge 1\\ \#\{i|k_i \ge 2\}=s}}x^{k_1-1}y x^{k_2-1}y \cdots x^{k_n-1}y.
\end{equation}

Following \cite{OZ}, we denote by $G_i(k,n,s;z)$ the sum of the MPLs with respect to $g_i(k,n,s)$; 
\begin{equation}
G_i(k,n,s;z) := \Li(g_i(k,n,s);z) = \sum_{\substack{w \in \fh^i\\|w|=k,\; d(w)=n\\ h(w)=s}} \Li(w;z).
\end{equation}
Especially, we get the following formula;
\begin{equation}
G_0(k,n,s;z) = \sum_{\substack{k_1+\cdots+k_n=k\\k_1 \ge 2,\; k_2,\ldots,k_n \ge 1\\ \#\{i|k_i \ge 2\}=s}}\Li_{k_1,\ldots,k_n}(z).
\end{equation}
Now the equation \eqref{MPL_diff_x} implies the following differential relation;
\begin{equation}
z\frac{d}{dz}G_0(k,n,s;z) = G_1(k-1,n,s;z).
\end{equation}

\subsection{The formal KZ equation} \label{subsec:Intro:formalKZ}
The theory of the formal KZ equation was established by Drinfel'd (\cite{D}). Here, as well as the previous subsection, we explain the relations between the formal KZ equation and the multiple polylogarithms of one variable to \cite{OU}.

\paragraph{The formal KZ equation}

Let $\fH = \C\langle \langle X,Y \rangle \rangle$ be an algebra of non-commu\-tative formal power series over $\C$ in letters $X,Y$. Then $\fh$ and $\fH$ have a dual Hopf algebra structures to each other. Let $H_0(z)$ be a $\fH$-valued analytic function defined by
\begin{equation}
H_0(z) := \sum_{w} \Li(w;z)W, \label{KZsol}
\end{equation}
where the summation in $w$ runs over the set of words in $\fh$ and $W$ stands for the capitalization of $w$, that is, the word in $\fH$ corresponding to $w$. The function $H_0(z)$ is the unique solution to the the formal KZ equation;
\begin{equation}
\frac{dG}{dz} = \left(\frac{X}{z}+\frac{Y}{1-z}\right)G, \label{KZeq}
\end{equation}
which satisfies the asymptotic property: $H_0(z) z^{-X} \to 1 \;\; (z \to 0)$. Furthermore one can write the inverse of $H_0(z)$ by
\begin{equation}
H_0(z)^{-1} = \sum_{w} \Li(S(w);z)W,
\end{equation}
where $S$ is the anti-automorphism on $\fh$ defined by $S(x)=-x$ and $S(y)=-y$. We note that the anti-automorphism $S$ is the antipode of $\fh$ as a Hopf algebra.

\paragraph{Representations of the formal KZ equation}

Let $\fH' = \C\langle X,Y \rangle$ be an algebra of non-commutative polynomial over $\C$ in letters $X,Y$ (that is, $\fH$ is a completion of $\fH'$). For a given representation $\rho$ of $\fH'$, $\rho:\fH' \to M(n,\C)$, the equation
\begin{equation}
\frac{dG}{dz} = \left(\frac{\rho(X)}{z}+\frac{\rho(Y)}{1-z}\right)G \label{KZeq_rep}
\end{equation}
is called the representation of the formal KZ equation by $\rho$ and the formal sum
\begin{equation}
\rho(H_0(z)) := \sum_{w} \Li(w;z)\rho(W) \label{KZsol_rep}
\end{equation}
is called the representation of the fundamental solution $H_0(z)$. In general $\rho(H_0(z))$ is not a $M(n,\C)$-valued analytic function, however, if it is analytic (that is, the formal sum \eqref{KZsol_rep} converges absolutely), the function \eqref{KZsol_rep} gives a solution to the equation \eqref{KZeq_rep}. Then the function \eqref{KZsol_rep} can be viewed as a solution to the equation \eqref{KZeq_rep} obtained by successive integral over the segment $[0,z]$.

\subsection{The Gauss hypergeometric equation and its solutions} \label{subsec:Intro:HGE}

Let $\alpha,\beta$ and $\gamma$ be a complex parameters and $\rho_0: \fH' \to M(2,\C)$ be a representation defined by
\begin{equation}
\rho_0(X)=\begin{pmatrix}0&\beta\\0&1-\gamma\end{pmatrix},\;\; \rho_0(Y)=\begin{pmatrix}0&0\\\alpha&\alpha+\beta+1-\gamma\end{pmatrix}. \label{XYdef}
\end{equation}
Then the representation of the formal KZ equation by $\rho_0$ is the Gauss hypergeometric equation;
\begin{equation}
z(1-z)\frac{d^2w}{dz^2} + (\gamma - (\alpha+\beta+1)z) \frac{dw}{dz} - \alpha \beta w = 0. \label{HGeq}
\end{equation}
Indeed one can rewrite the equation to the system;
\begin{equation}
\frac{d}{dz}\vect{v_1}{v_2} = \left(\frac{1}{z}\begin{pmatrix}0&\beta\\0&1-\gamma\end{pmatrix}+\frac{1}{1-z}\begin{pmatrix}0&0\\\alpha&\alpha+\beta+1-\gamma\end{pmatrix}\right)\vect{v_1}{v_2}, \label{HGeq2}
\end{equation}
where $v_1=w$ and $v_2=\frac{1}{\beta}z\frac{dw}{dz}$.

The Gauss hypergeometric equation is a Fuchsian equation of the second order with three regular singular points $0,1$ and $\infty$ in $\bP^1$.

In what follows, we assume that the parameters $\alpha,\beta,\gamma$ and $\gamma-\alpha-\beta$ are not integers. Under this assumption, the fundamental solution matrices $\Phi_i$ of the equation \eqref{HGeq2} in the neighborhood of $z=i \; (i=0,1,\infty)$ are given by the following (\cite{WW}).
\begin{equation}
\Phi_i=\begin{pmatrix}\varphi^{(i)}_0&\varphi^{(i)}_1\\\frac{1}{\beta}z\frac{d}{dz}\varphi^{(i)}_0&\frac{1}{\beta}z\frac{d}{dz}\varphi^{(i)}_1\end{pmatrix}, \label{def_Phi}
\end{equation}
where the functions $\varphi^{(i)}_j(z)\;\; (i=0,1,\infty,\; j=0,1)$ stand for the solutions in the neighborhood of $z=i$ defined by
\begin{align}
\varphi^{(0)}_0(z)&= F(\alpha,\beta,\gamma;z) = \sum_{n=0}^{\infty} \frac{(\alpha)_n(\beta)_n}{(\gamma)_n n!} z^n, \label{HGeq_sol_0_0}\\
\varphi^{(0)}_1(z)&=z^{1-\gamma}F(\alpha+1-\gamma,\beta+1-\gamma,2-\gamma;z), \label{HGeq_sol_0_1-gamma}\\
\varphi^{(1)}_0(z)&=F(\alpha,\beta,\alpha+\beta+1-\gamma;1-z), \label{HGeq_sol_1_0}\\
\varphi^{(1)}_1(z)&=(1-z)^{\gamma-\alpha-\beta}F(\gamma-\alpha,\gamma-\beta,\gamma-\alpha-\beta+1;1-z), \label{HGeq_sol_1_gamma-alpha-beta}\\
\varphi^{(\infty)}_0(z)&=z^{-\alpha}F(\alpha,\alpha+1-\gamma,\alpha-\beta+1;1/z), \label{HGeq_sol_infty_alpha}\\
\varphi^{(\infty)}_1(z)&=z^{-\beta}F(\beta,\beta+1-\gamma,\beta-\alpha+1;1/z), \label{HGeq_sol_infty_beta}
\end{align}
where we define the branch of these complex power by the principal values. One can continue these solutions analytically to $\bP^1-\{0,1,\infty\}$ as many-valued function. The connection matrices of these fundamental solution matrices are given by the following formula via the Euler or Barnes integral expression of hypergeometric function;
\begin{align}
\Phi_1^{-1}\Phi_0 & = C^{01} =\begin{pmatrix} \frac{\ds \Gamma(\gamma)\Gamma(\gamma-\alpha-\beta)}{\ds \Gamma(\gamma-\alpha)\Gamma(\gamma-\beta)} & \frac{\ds \Gamma(2-\gamma)\Gamma(\gamma-\alpha-\beta)}{\ds \Gamma(1-\alpha)\Gamma(1-\beta)} \\ \frac{\ds \Gamma(\gamma)\Gamma(\alpha+\beta-\gamma)}{\ds \Gamma(\alpha)\Gamma(\beta)} & \frac{\ds \Gamma(2-\gamma)\Gamma(\alpha+\beta-\gamma)}{\ds \Gamma(\alpha+1-\gamma)\Gamma(\beta+1-\gamma)} \end{pmatrix}, \label{connection01}\\
\Phi_{\infty}^{-1} \Phi_0 & = C^{0\infty} =\begin{pmatrix} e^{-\pi i \alpha} \frac{\ds \Gamma(\gamma)\Gamma(\beta-\alpha)}{\ds \Gamma(\beta)\Gamma(\gamma-\alpha)} & e^{\pi i (\gamma-\alpha-1)} \frac{\ds \Gamma(2-\gamma)\Gamma(\beta-\alpha)}{\ds \Gamma(\beta+1-\gamma)\Gamma(1-\alpha)} \\ e^{-\pi i \beta} \frac{\ds \Gamma(\gamma)\Gamma(\alpha-\beta)}{\ds \Gamma(\alpha)\Gamma(\gamma-\beta)} & e^{\pi i (\gamma-\beta-1)} \frac{\ds \Gamma(2-\gamma)\Gamma(\alpha-\beta)}{\ds \Gamma(\alpha+1-\gamma)\Gamma(1-\beta)} \end{pmatrix} \label{connection0infty}.
\end{align}

In what follows, the hypergeometric equation (function, series, \ldots) simply means the Gauss hypergeometric equation (function, series, \ldots, respectively).

\section{Analytic properties of the MPLs}\label{sec:Pre:Anal_MPL}

In this section, we discuss analytic continuation of the multiple polylogarithms of one variable to the universal covering space of $\bP^1-\{0,1,\infty\}$ to prove the analytic property of the representation by $\rho_0$ of the solution to the formal KZ equation.

\begin{prop} \label{MPLestimate}
Let $\fU$ be the universal covering space of $\bP^1-\{0,1,\infty\}$ and $K$ be a compact subset of $\fU$. There exists a constant $M_K$ depending only on $K$ such that, for any word $w \in \fh$,
\begin{equation}
|\Li(w;z)|<M_K \quad \forall z \in K. \label{MPLupperbound}
\end{equation}
\end{prop}

A basic idea to prove this lemma is due to Lappo-Danilevsky. We try to modify the theory on \cite{LD} p.159-163.

Let $\pi:\fU \to \bP^1-\{0,1,\infty\}$ be the canonical projection and the real interval $\fI$ be a simply-connected subset of $\fU$ defined by $\fI=\{z \in \fU\;|\;0<\Re \pi(z)<1, \Im \pi(z)=0, \arg(z)=\arg(1-z)=0\}$. Thus, for all word $w$ in $\fh^1$ and $z \in \fI$, $\Li(w;z)$ has an expansion \eqref{MPLdef} and $\log z$ has an expansion $\log z=-\sum_{n=1}^{\infty}\frac{(1-z)^n}{n}$. We note that, if $z \in \fI$, $\Li(w;z)$ converges to 0 as $z$ tends to 0 and $\log z$ converges to 0 as $z$ tends to 1.

Let $z,p$ be points of $\fU$ and $C_p^z$ be a path on $\fU$ from $p$ to $z$. For any word $w \in \fh$, we define the MPLs with prescribing an initial point and a path by an inductive way such as
\begin{align}
\Li_{p,C_p^z}(xw;z)&=\int_{p,C_p^z}^z \frac{\Li_{p,C_p^z}(w;z)}{z}dz,\\
\Li_{p,C_p^z}(yw;z)&=\int_{p,C_p^z}^z \frac{\Li_{p,C_p^z}(w;z)}{1-z}dz,\\
\Li_{p,C_p^z}(1;z)&=1.
\end{align}

These MPLs satisfy the following properties.
\begin{lem}{(\cite{LD})} \label{lem:Lappo}
\begin{enumerate}
\item 
For any word $w \in \fh$ and points $p,z \in \fU$, the value of $\Li_{p,C_p^z}(w;z)$ does not depend on choice of a path $C_p^z$ on $\fU$.
\item
Let $\delta$ be the distance between $\pi(C_z^p)$ and $\{0,1\}$;
\begin{equation}
\delta = \dist(\pi(C_z^p),\{0,1\}) = \inf_{\substack{z_1 \in \pi(C_z^p)\\ z_2 \in \{0,1\}}}|z_1-z_2|,
\end{equation}
and $\sigma$ be the length of the path $\pi(C_p^z)$. Then we have
\begin{equation}
|\Li_{p,C_p^z}(w;z)| < \frac{1}{|w|!}\left(\frac{\sigma}{\delta}\right)^{|w|}. \label{HLestimate}
\end{equation}

\end{enumerate}
\end{lem}

The next lemma, which is also due to \cite{LD}, follows from the coproduct structure of $\fh$ as a Hopf algebra.
\begin{lem}
Let $w = a_1a_2\cdots a_r$ be a word in $\fh^1$, where each $a_i$ denotes the letter $x$ or $y$ ($i=1,\ldots,r,\; a_r=y$), and $C_p^z$ be a path from $p$ to $z$ on $\fU$. Choosing a point $q$ on $C_p^z$, we divide the path $C_p^z$ as $C_p^z=C_p^q+C_q^z$. Then $\Li_{p,C_p^z}(w;z)$ satisfies
\begin{equation}
\Li_{p,C_p^z}(w;z)=\sum_{i=0}^{r} \Li_{q,C_q^z}(a_1\cdots a_i;z)\Li_{p,C_p^q}(a_{i+1}\cdots a_r;q). \label{HLconnect}
\end{equation}
Here we use a convention such as $\Li_{q,C_q^z}(a_1\cdots a_i;z)=1$ for $i=0$ and\\$\Li_{p,C_p^q}(a_{i+1}\cdots a_r;q)=1$ for $i=r$.
\end{lem}

Clearly, for any word $w \in \fh^1$ and $p \in \fI$, the MPL $\Li(w;z)$ defined in \S \ref{sec:Introduction} satisfy
\begin{equation}
\Li(w;z) = \lim_{\substack{\varepsilon \to 0\\ \varepsilon \in \fI}} \Li_{\varepsilon,[\varepsilon,p]+C_p^z}(w;z), \label{MPL_p_to_0}
\end{equation}
where $[\varepsilon,p]$ stands for the path from $\varepsilon$ to $p$ on the interval $\fI$.

One can prove the following lemma by induction on the length of the word $w$.
\begin{lem}
For any $z \in (0,\frac{1}{2}] \subset \R$ and any word $w \in \fh^1$,
\begin{equation}
|\Li(w;z)| \le 1. \label{MPL_series_bound}
\end{equation}
\end{lem}

Using these lemmas, we prove the proposition.

\begin{proof}[Proof of Proposition \ref{MPLestimate}]

Let $p(K)$ and $l(K)$ be constants depending only on $K$ such as
\begin{align}
p(K) &= \min\{\dist(\pi(K),\{0,1\}),\; \frac{1}{2}\},\\
\ds l(K)&=\max_{z \in K} \left(\inf_{\substack{\text{$C$ : path from $p(K)$ to $z$ on $\fU$}\\\dist(\pi(C),\{0,1\}) \ge p(K)}} \length(\pi(C))\right). \label{def_l(K)}
\end{align}
In the right hand side of \eqref{def_l(K)}, we identify the number $p(K) \in (0,1)$ with a point in $\fI \subset \fU$.

By virtue of \eqref{HLconnect} and \eqref{MPL_p_to_0}, for any word $w = a_1a_2\cdots a_r \in \fh^1$ we obtain
\begin{equation}
\Li(w;z)=\sum_{i=0}^{r} \Li_{p(K),C_{p(K)}^z}(a_1\cdots a_i;z)\Li(a_{i+1}\cdots a_r;p(K)).
\end{equation}
Consequently, by Lemma \ref{lem:Lappo} and \eqref{MPL_series_bound}, we have
\begin{align}
|\Li(w;z)|&<\sum_{i=0}^{r} |\Li_{p(K),C_{p(K)}^z}(a_1\cdots a_i;z)||\Li(a_{i+1}\cdots a_r;p(K))| \notag \\
&<\sum_{i=0}^{r} \frac{1}{i!}\left(\frac{l(K)}{p(K)}\right)^i<\exp\left(\frac{l(K)}{p(K)}\right).
\end{align}

Furthermore, for any word $w \in \fh^1$,
\begin{align}
|\Li(w x^n;z)| &< \sum_{j=0}^n |\Li(\reg^1(w x^{n-j});z)|\frac{|\log z|^j}{j!} \notag \\
&< \exp\left(\frac{l(K)}{p(K)}\right) \exp\left(\max_{z \in K}|\log z|\right).
\end{align}
Thus the estimation \eqref{MPLupperbound} holds for $\ds M_K=\exp\left(\frac{l(K)}{p(K)} + \max_{z \in K}|\log z|\right)$.

\end{proof}

\section{Representation of the hypergeometric function by the MPLs} \label{sec:show_HGF_by_MPL}

In this section, we give a concrete form of $\rho_0(H_0(z))$ the representation of the solution to the formal KZ equation and expand the hypergeometric function as a series of the the multiple polylogarithms of one variable.

\subsection{Main Result} \label{subsec:show_HGF_by_MPL:MainThorem1}

\begin{thm}\label{thm:MainTheorem1}
Assume that $|1-\gamma|,|\alpha+1-\gamma|,|\beta+1-\gamma|$ and $|\alpha+\beta+1-\gamma|<\frac{1}{2}$. Then we have the expression of $\varphi^{(0)}_0(z)$, the hypergeometric function, as follows;
\begin{multline}
\varphi^{(0)}_0(z) = 1+ \alpha\beta \sum_{\substack{k,n,s > 0\\k\ge n+s\\n\ge s}} G_0(k,n,s;z) \\
\times (1-\gamma)^{k-n-s}(\alpha+\beta+1-\gamma)^{n-s}((\alpha+1-\gamma)(\beta+1-\gamma))^{s-1}.
\end{multline}
\noindent The series in the right hand side converges absolutely and uniformly on any compact subset $K$ of the universal covering space $\fU$.
\end{thm}

We prove this theorem in the following. We set $p=1-\gamma$ and $q=\alpha+\beta+1-\gamma$.

\subsection{The image of word in $\fH$ by the representation $\rho_0$}\label{subsec:show_HGF_by_MPL:Rep_of_formalKZ_0}

We consider the representation $\rho_0: \fH' \to M(2,\C)$ given by \eqref{XYdef};
\begin{equation}
\rho_0(X)=\begin{pmatrix}0&\beta\\0&p\end{pmatrix},\;\; \rho_0(Y)=\begin{pmatrix}0&0\\\alpha&q\end{pmatrix}.
\end{equation}
The representation of the formal KZ equation \eqref{KZeq} by $\rho_0$
\begin{equation}
\frac{d}{dz}G = \left(\frac{\rho_0(X)}{z}+\frac{\rho_0(Y)}{1-z}\right)G \label{HGeq3}
\end{equation}
is nothing but the hypergeometric equation \eqref{HGeq2} and the representation of the solution $H_0(z)$
\begin{equation}
\rho_0(H_0(z)) = \sum_{w} \Li(w;z)\rho_0(W) \label{rep_H0}
\end{equation}
gives , if it converges absolutely, the fundamental solution matrix which has the asymptotic property $\rho_0(H_0) z^{-\rho_0(X)} \to I \quad (z \to 0)$. In order to compute the series of the right hand side of \eqref{rep_H0} and to show its convergence, we prepare the following lemma.

\begin{lem} \label{rho0calc}
For any word $W$ in $\fH$, we define the weight $|W|$, the depth $d(W)$ and the height $h(W)$ of $W$ in a similar fashion as in $\fh$;
\begin{align*}
|W| &:= \text{the number of letters in } W,\\
d(W) &:= \text{the number of $Y$ which appears in $W$},\\
h(W) &:= \text{(the number of $YX$ which appears in $W$)} + 1.
\end{align*}
Then for any non-empty word $W \in \fH$, the representation $\rho_0(W)$ is given by
\begin{align}
\rho_0(W)&=p^{|W|-d(W)-h(W)} q^{d(W)-h(W)} (\alpha\beta+pq)^{h(W)-1} M, \label{rho0(W)explicit}
\intertext{where}
M&=\begin{cases}
\begin{pmatrix}\alpha\beta&\beta q\\ \alpha p&pq\end{pmatrix}&(\text{if } W \in X \fH Y)\\
\begin{pmatrix}0&0\\ \alpha p&pq\end{pmatrix}&(\text{if } W \in Y \fH Y \text{or\;} W=Y)\\
\begin{pmatrix}0&\beta q\\ 0&pq\end{pmatrix}&(\text{if } W \in X \fH X \text{or\;} W=X)\\
\begin{pmatrix}0&0\\ 0&pq\end{pmatrix}&(\text{if } W \in Y \fH X).
\end{cases}
\end{align}

\end{lem}

\begin{proof}
This is proved by straightforward computation.
\end{proof}

By putting $\delta=\max(|\alpha\beta|,|\alpha p|, |\beta q|, |pq|, 1)$, we obtain the following corollaries.

\begin{cor} \label{rho0Westimate}
For any word $W \in \fH$, there exists a constant $\delta$ depending only $\alpha,\beta,p,q$ such that
\begin{align}
||\rho_0(W)|| &\le \delta |p|^{|W|-d(W)-h(W)}|q|^{d(W)-h(W)}|\alpha\beta+pq|^{h(W)-1},
\end{align}
where $||A||$ denotes the maximal norm of a matrix $A=(a_{ij})_{1\le i,j \le 2}$, namely $||A||=\max_{i,j}|a_{ij}|$.
\end{cor}

\begin{cor} \label{rho0H0estimate}
If $|1-\gamma|,|\alpha+1-\gamma|,|\beta+1-\gamma|,|\alpha+\beta+1-\gamma| < \frac{1}{2}$, the representation $\rho_0(H_0(z))$ converges absolutely and uniformly on any compact subset $K$ of the universal covering $\fU$ of $\bP^1-\{0,1,\infty\}$.
\end{cor}

\begin{proof}
Let $K$ be a compact subset of $\fU$. By using Proposition \ref{MPLestimate} and the fact that the number of words in $\fh$ with fixed weight $k$ is $2^k$, we can show that there exists a constant $M_K$ depending only on $K$ such that
\begin{equation}
\left|\sum_{\substack{w \in \fh \text{: word},\;\; |w|=k\\d(w)=n,\;\; h(w)=s}}\Li(w;z)\right|<2^k M_K \qquad \forall z \in K.
\end{equation}
By Corollary \ref{rho0Westimate} we have
\begin{align}
||\rho_0&(H_0(z))||=||\sum_{w} \Li(w;z)\rho_0(W)|| \notag \\
&\le 1+\delta \sum_{k,n,s} \sum_{\substack{w\\|w|=k\\d(w)=n\\h(w)=s}}\Li(w;z)|p|^{|W|-d(W)-h(W)}|q|^{d(W)-h(W)}|\alpha\beta+pq|^{h(W)-1} \notag \\
&\le 1+\delta M_K \sum_{k,n,s} 2^k |p|^{k-n-s}|q|^{n-s}|\alpha\beta+pq|^{s-1} \notag \\
&< 1+4 \delta M_K \sum_{\substack{k,n,s >0\\k \ge n+s\\n \ge s}} (2|1-\gamma|)^{k-n-s}(2|\alpha+\beta+1-\gamma|)^{n-s} \notag \\
&\hspace{5cm}\times \big((2|\alpha+1-\gamma|)(2|\beta+1-\gamma|)\big)^{s-1}.
\end{align}
Hence  the series $\rho_0(H_0(z))$ converges absolutely and uniformly on $K$ if $|1-\gamma|,|\alpha+1-\gamma|,|\beta+1-\gamma|,|\alpha+\beta+1-\gamma| < \frac{1}{2}$.
\end{proof}

\subsection{The asymptotic properties of $\rho_0(H_0)$ and $\Phi_0$} \label{subsec:show_HGF_by_MPL:asymptotic}

From the discussion above, it follows that the representation $\rho_0(H_0(z))$ is the fundamental solution matrix of the hypergeometric equation on $\fU$. Hence there exists a linear relation between $\rho_0(H_0(z))$ and $\Phi_0$, which is also a fundamental solution matrix in the neighborhood of $z=0$ defined by \eqref{def_Phi}, as follows.

\begin{lem} \label{lem:Phi0_rho0(H0)}
\begin{equation}
\Phi_0(z)=\rho_0(H_0(z))\begin{pmatrix}1&1\\0&\frac{p}{\beta}\end{pmatrix}. \label{Phi0_rho0(H0)}
\end{equation}
\end{lem}

\begin{proof}
Let $D$ be a domain in $\C$ defined by $D=\{z \in \C|\; |z|<1\}-\{\Re z\le 0, \Im z=0\}$ and specify branches of all MPLs, which appears in $\rho_0(H_0(z))$, in $D$ by the expansion \eqref{MPLdef} and a branch of $\log z$ by the principal value (that is, $D$ is a domain in $\fU$ which includes the interval $\fI$ and $\pi(D)$ is simply-connected). It is enough to prove that both sides of \eqref{Phi0_rho0(H0)} have the same asymptotic property as $z$ tends to 0 in $D$.

By virtue of \S \ref{subsec:Intro:formalKZ}, clearly $\rho_0(H_0(z)) z^{-\rho_0(X)} \to I\;(z \to 0)$ holds. On the other hand, because of the formula
\begin{equation}
z^{-\rho_0(X)}=\begin{pmatrix}1&\frac{\beta}{p}(z^{-p}-1)\\0&z^{-p}\end{pmatrix},
\end{equation}
we get
\begin{align}
\Phi_0(z)&\begin{pmatrix}1&1\\0&\frac{p}{\beta}\end{pmatrix}^{-1} z^{-\rho_0(X)} \notag \\
&= \begin{pmatrix}F_{00}(z)&\frac{\beta}{p}(F_{01}(z)-F_{00}(z))\\ \frac{1}{\beta}zF_{00}'(z)& \frac{1}{p}(pF_{01}(z)+zF'_{01}(z)-zF'_{00}(z))\end{pmatrix} \notag \to I \quad (z \to 0),
\end{align}
where $F_{00}$ and $F_{01}$ are introduced through $\varphi^{(0)}_0=F_{00},\; \varphi^{(0)}_1=z^{1-\gamma}F_{01}$. Hence $\Phi_0$ and $\rho_0(H_0(z))$ are the same solution to the hypergeometric equation.
\end{proof}

\subsection{Proof of Theorem \ref{thm:MainTheorem1}} \label{subsec:show_HGF_by_MPL:proof}

Now we prove Theorem \ref{thm:MainTheorem1} by using the preceding lemmas.

\begin{proof}[Proof of Theorem \ref{thm:MainTheorem1}]

From Corollary \ref{rho0H0estimate}, if $|1-\gamma|,|\alpha+1-\gamma|,|\beta+1-\gamma|,|\alpha+\beta+1-\gamma| < \frac{1}{2}$, any matrix element of $\rho_0(H_0(z))$ converges absolutely and uniformly on any compact subset $K$ of $\fU$. Since we see that $\varphi^{(0)}_0(z)=\left(\rho(H_0(z))\begin{pmatrix}1&1\\0&\frac{p}{\beta}\end{pmatrix}\right)_{11}$ satisfies the same property. Here $A_{ij}$ stands for the $(i,j)$ element of a matrix $A$. This is computed as follows;
\begin{align}
\varphi^{(0)}_0(z) &=1+ \alpha\beta \sum_{\substack{k,n,s >0\\k \ge n+s\\n \ge s}} G_0(k,n,s;z) p^{k-n-s}q^{n-s}(\alpha\beta+pq)^{s-1} \notag \\
=1+ &\alpha\beta \sum_{\substack{k,n,s >0\\k \ge n+s\\n \ge s}} G_0(k,n,s;z) \notag \\
&\qquad \quad \times (1-\gamma)^{k-n-s}(\alpha+\beta+1-\gamma)^{n-s}((\alpha+1-\gamma)(\beta+1-\gamma))^{s-1}. \label{Rep_phi_00}
\end{align}

In the discussion above, we have assumed that $1-\gamma, \alpha, \beta \neq 0$. However, the formula \eqref{Rep_phi_00} makes sense even if $1-\gamma \to 0, \alpha \to 0$ and $\beta \to 0$.

\end{proof}

\begin{cor}\label{cor:MainTheorem1}
The solution $\varphi^{(0)}_1(z)$ to the hypergeometric equation, which has the exponent $1-\gamma$ at $z=0$, is given as follows;
\begin{multline}
\varphi^{(0)}_1(z) = z^{1-\gamma}\bigg(1+(\alpha+1-\gamma)(\beta+1-\gamma)\sum_{k,n,s} G_0(k,n,s;z)\\
\phantom{z^{1-\gamma}\bigg(1+(\alpha+1-\gamma)}\times (\gamma-1)^{k-n-s}(\alpha+\beta+1-\gamma)^{n-s}(\alpha\beta)^{s-1} \bigg).
\end{multline}
\end{cor}

\paragraph{Remark}
This corollary is proved immediately by using of Theorem \ref{thm:MainTheorem1} and the formula $\varphi^{(0)}_1(z)=z^{1-\gamma}F(\alpha+1-\gamma,\beta+1-\gamma,2-\gamma;z)$. However one can also prove this as an algebraic way to compute $\ds \varphi^{(0)}_1(z)=\left(\rho_0(H_0(z))\begin{pmatrix}1&1\\0&\frac{p}{\beta}\end{pmatrix}\right)_{12}$.

\section{The connection relation between the regular solutions at $z=0$ and $z=1$} \label{sec:connection_0_1_regular}

In this section, we investigate various functional relations of the multiple polylogarithms of one variable considering the $(1,1)$-element of the relation formula \eqref{connection01}, and the relations of the multiple zeta values known as Ohno-Zagier relation(\cite{OZ}) by taking the limit as $z \to 1$.

For the purpose, we first express the inverse of $\Phi_1$ \eqref{def_Phi} as a series of the MPLs, and expand the gamma functions which appear in the connection matrix $C^{01}$ \eqref{connection01} as a series of the zeta values. It is too difficult to calculate the whole of the latter series, so that we consider its degenerate form by specializing the parameters.

\subsection{The inverse of the fundamental solution matrix in the neighborhood of $z=1$} \label{subsec:connection_0_1_regular:Phi_1^-1}

Let $G$ be a solution to the hypergeometric equation \eqref{HGeq3}, then ${}^tG^{-1}$ satisfies
\begin{equation}
\frac{d}{dt}{}^tG^{-1} = \left(\frac{{}^t\rho_0(Y)}{t}+\frac{{}^t\rho_0(X)}{1-t}\right){}^tG^{-1},
\end{equation}
where $t=1-z$. So define a representation $\rho_1: \fH' \to M(2,\C)$ such as
\begin{equation}
\rho_1(X)={}^t\rho_0(Y)=\begin{pmatrix}0&\alpha\\0&q\end{pmatrix},\;\; \rho_1(Y)={}^t\rho_0(X)=\begin{pmatrix}0&0\\\beta&p\end{pmatrix}.
\end{equation}
Then the matrix-valued function ${}^t\Phi_1^{-1}$ is a fundamental solution matrix of the representation of the formal KZ equation \eqref{KZeq} by $\rho_1$
\begin{equation}
\frac{d}{dt}G = \left(\frac{\rho_1(X)}{t}+\frac{\rho_1(Y)}{1-t}\right)G. \label{HGeq_1_inv}
\end{equation}
On the other hand, the representation $\rho_1(H_0(t))$ is also a fundamental solution matrix of \eqref{HGeq_1_inv}. This representation is nothing but $\rho_0(H_0(z))$ up to changing the variable $z \to t$ and the parameters $(\alpha,\beta,p,q) \to (\beta,\alpha,q,p)$. Similarly as in Lemma \ref{lem:Phi0_rho0(H0)}, we obtain the linear relation
\begin{equation}
{}^t\Phi_1^{-1}=\rho_1(H_0(t))\begin{pmatrix}1&\frac{\alpha\beta}{(\alpha+\beta-\gamma)q}\\0&\frac{1}{\alpha+\beta-\gamma}\end{pmatrix}.
\end{equation}

Therefore the $(1,1)$ and $(2,1)$-elements of ${}^t\Phi_1^{-1}$ lead to the following proposition.

\begin{prop} \label{prop:Phi1inv_regular}
Assume that $|1-\gamma|,|\alpha+1-\gamma|,|\beta+1-\gamma|$ and $|\alpha+\beta+1-\gamma|<\frac{1}{2}$. The $(1,1)$ and $(1,2)$-elements of $\Phi_1^{-1}$ are expressed as follows;
\begin{align}
\begin{split}
(\Phi_1^{-1})_{11} &= 1+ \alpha\beta \sum_{\substack{k,n,s > 0\\k\ge n+s\\n\ge s}} G_0(k,k-n,s;1-z) \\
&\times (1-\gamma)^{k-n-s}(\alpha+\beta+1-\gamma)^{n-s}((\alpha+1-\gamma)(\beta+1-\gamma))^{s-1}, 
\end{split}\\
\begin{split}
(\Phi_1^{-1})_{12} &= \beta \sum_{\substack{k,n,s > 0\\k\ge n+s\\n\ge s}} G_1(k-1,k-n,s;1-z) \\
&\times (1-\gamma)^{k-n-s}(\alpha+\beta+1-\gamma)^{n-s}((\alpha+1-\gamma)(\beta+1-\gamma))^{s-1}.
\end{split}
\end{align}
\noindent The series of the right hand sides converge absolutely and uniformly on any compact subset $K$ of $\fU$.
\end{prop}

\subsection{The expansion of the connection matrix as a series of the zeta values} \label{subsec:connection_0_1_regular:gamma_expand}

To compare both sides of the connection relation \eqref{connection01}, we expand the $(1,1)$-element of the connection matrix $(C^{01})_{11}=\frac{\Gamma(\gamma)\Gamma(\gamma-\alpha-\beta)}{\Gamma(\gamma-\alpha)\Gamma(\gamma-\beta)}$ as a series of the zeta values. We set $p=1-\gamma, q=\alpha+\beta+1-\gamma$ and $r=(\alpha+1-\gamma)(\beta+1-\gamma)$.

The following formula for the gamma function is famous(\cite{WW});
\begin{equation}
\frac{1}{\Gamma(1-z)}=\exp(-cz-\sum_{n=2}^{\infty}\frac{\zeta(n)}{n}z^n), \label{Gamma_expand}
\end{equation}
where the complex number $c$ is the Euler constant, namely\\ $c=\lim_{n\to \infty}(\sum_{k=1}^n \frac{1}{k}-\log n)$.

For a sequence of complex numbers $\text{\boldmath $a$}=(a_1,a_2,\ldots)$ and a non-negative integer $n$, we introduce the Schur polynomial $P_n(\text{\boldmath $a$})$ through the following generating function;

\begin{equation}
\exp(\sum_{n=1}^{\infty} a_n z^n)=\sum_{n=0}^{\infty}P_n(\text{\boldmath $a$})z^n, \label{def_Schur}
\end{equation}
that is
\begin{align}
P_0(\text{\boldmath $a$})&=1, \notag \\
P_n(\text{\boldmath $a$})&=\sum_{k_1+2k_2+3k_3+\cdots=n}\frac{a_1^{k_1}}{k_1!}\frac{a_2^{k_2}}{k_2!}\frac{a_3^{k_3}}{k_3!}\cdots \quad (n \ge 1).
\end{align}

We also define the integers $N_{i,j}^{(n)}$ as 
\begin{equation}
\begin{array}{ll}
N_{i,j}^{(n)}=0, & (i<0 \text{ or } j<0)\\
N_{0,0}^{(0)}=1, & (i=j=0)\\
a^n+b^n=\sum_{i,j}N_{i,j}^{(n)}(a+b)^i(ab)^j. &(\text{otherwise}) \label{def_N}
\end{array}
\end{equation}
\noindent Since $a^n+b^n$ is a symmetric polynomial of $a$ and $b$, this definition is well-defined and we have $N_{i,j}^{(n)}=0$ if $i+2j\neq n$. We denote by $N_{i,j}=N_{i,j}^{(i+2j)}$.

Under these notations, one can show the following lemma and proposition.
\begin{lem} \label{lem:A(a)A(b)expand}
In the algebra of formal power series $\C[[a,b]]$,
\begin{equation}
\left(\sum_{i=0}^{\infty}A_i a^i\right)\left(\sum_{i=0}^{\infty}A_i b^i\right)=\sum_{k,l=0}^{\infty}\left(\sum_{i=0}^l A_iA_{2l+k-i}N_{k,l-i}\right)(a+b)^k(ab)^l.
\end{equation}
\end{lem}

\begin{prop} \label{prop:Gamma_expand}
For the sequence $\bzeta=(0,\frac{\zeta(2)}{2},\frac{\zeta(3)}{3},\frac{\zeta(4)}{4},\cdots)$ and the complex numbers $p=1-\gamma, q=\alpha+\beta+1-\gamma$ and $r=(\alpha+1-\gamma)(\beta+1-\gamma)$, we obtain the following expansion;
\begin{multline}
\frac{\Gamma(\gamma)\Gamma(\gamma-\alpha-\beta)}{\Gamma(\gamma-\alpha)\Gamma(\gamma-\beta)} =\\
\sum_{k,l,m=0}^{\infty}\left( \sum_{i=0}^k \sum_{j=0}^l \sum_{\mu=0}^m \binom{i+j}{i}P_{k-i}(\bzeta)P_{l-j}(\bzeta)P_{\mu}(-\bzeta)P_{i+j+2m-\mu}(-\bzeta)N_{i+j,m-\mu}\right) \\
\times p^k q^l r^m. \label{Gamma_expand}
\end{multline}
\end{prop}

\begin{proof}
\begin{align*}
&\frac{\Gamma(\gamma)\Gamma(\gamma-\alpha-\beta)}{\Gamma(\gamma-\alpha)\Gamma(\gamma-\beta)}=\frac{\Gamma(1-p)}{e^{c(1-p)}}\frac{\Gamma(1-q)}{e^{c(1-q)}}\frac{e^{c(1-(\alpha+1-\gamma))}}{\Gamma(1-(\alpha+1-\gamma))}\frac{e^{c(1-(\beta+1-\gamma))}}{\Gamma(1-(\beta+1-\gamma)} \\
&=\exp(\sum_{n=2}^{\infty}\frac{\zeta(n)}{n}p^n)\exp(\sum_{n=2}^{\infty}\frac{\zeta(n)}{n}q^n)\\
&\qquad \qquad \times \exp(-\sum_{n=2}^{\infty}\frac{\zeta(n)}{n}(\alpha+1-\gamma)^n)\exp(-\sum_{n=2}^{\infty}\frac{\zeta(n)}{n}(\beta+1-\gamma)^n) \\
&=\sum_{n=0}^{\infty}P_n(\zeta)p^n \sum_{n=0}^{\infty}P_n(\zeta)q^n \sum_{n=0}^{\infty}P_n(-\zeta)(\alpha+1-\gamma)^n \sum_{n=0}^{\infty}P_n(-\zeta)(\beta+1-\gamma)^n.
\end{align*}
Here by making use of Lemma \ref{lem:A(a)A(b)expand}, we have
\begin{align*}
&\frac{\Gamma(\gamma)\Gamma(\gamma-\alpha-\beta)}{\Gamma(\gamma-\alpha)\Gamma(\gamma-\beta)}=\sum_{n=0}^{\infty}P_n(\zeta)p^n \sum_{n=0}^{\infty}P_n(\zeta)q^n \\
&\qquad \times \sum_{k,l,m}\binom{k+l}{k}\left(\sum_{\mu=0}^{m}P_{\mu}(-\bzeta)P_{k+l+2m-\mu}(-\bzeta)N_{k+l,m-\mu}\right)p^k q^l r^m \\
&\qquad =\text{the right hand side of \eqref{Gamma_expand}}.
\end{align*}
\end{proof}

\subsection{Functional relations obtained from the $(1,1)$-element of the connection relation \eqref{connection01}} \label{subsec:connection_0_1_regular:MPLrel_01}

Expanding the $(1,1)$-element of \eqref{connection01}
\begin{equation}
(\Phi_1^{-1})_{11}(\Phi_0)_{11} + (\Phi_1^{-1})_{12}(\Phi_0)_{21} = (C^{01})_{11} \label{connection01_11}
\end{equation}
to series in $p,q,r$, and comparing coefficients of $p^k q^l r^m$ in both sides, we obtain the following theorem.

\begin{thm} \label{thm:MPLrel_01}
Put $\bar{G}_i(k,n,s;z)=G_i(k,n,s;z)-G_i(k,n,s+1;z) \quad (i=0,1)$. Then we have
\begin{align}
&\bar{G}_0(k+l+2m,l+m,m;z)+\bar{G}_0(k+l+2m,k+m,m;1-z) \notag \\*
&+\sum_{\substack{k'+k''=k\\l'+l''=l\\m'+m''=m}}\Big( \bar{G}_0(k'\!\!\!+\!l'\!\!\!+\!2m',l'\!\!\!+\!m',m';z)\bar{G}_0(k''\!\!\!+\!l''\!\!\!+\!2m'',k''\!\!\!+\!m'',m'';1-z) \notag \\*
&\quad +\bar{G}_1(k'\!\!\!+\!l'\!\!\!+\!2m'\!\!\!-\!1,l'\!\!\!+\!m',m';z)G_1(k''\!\!\!+\!l''\!\!\!+\!2m''\!\!\!+\!1,k''\!\!\!+\!m''\!\!\!+\!1,m''\!\!\!+\!1;1-z)\Big) \notag \\*
&= \sum_{i=0}^k \sum_{j=0}^l \sum_{\mu=0}^m \binom{i+j}{i}P_{k-i}(\bzeta)P_{l-j}(\bzeta)P_{\mu}(-\bzeta)P_{i+j+2m-\mu}(-\bzeta)N_{i+j,m-\mu}. \label{MPLrel_01}
\end{align}
\end{thm}

Since $\lim_{z \to 1}G_0(k,n,s;z) = G_0(k,n,s;1),\; \lim_{z \to 1}G_0(k,n,s;1-z) = 0,$ and $\lim_{z \to 1}G_1(k,n,s;z)G_1(k',n',s';1-z) = 0$, the limit of the equation \eqref{MPLrel_01} as $z \to 1$ implies the following corollary, which was originally shown by \cite{OZ}.

\begin{cor} \label{cor:MPLrel_01:Ohno-Zagier}
\begin{align}
&\bar{G}_0(k+l+2m,l+m,m;1)=\bar{G}_0(k+l+2m,k+m,m;1) \notag \\
&= \sum_{i=0}^k \sum_{j=0}^l \sum_{\mu=0}^m \binom{i+j}{i}P_{k-i}(\bzeta)P_{l-j}(\bzeta)P_{\mu}(-\bzeta)P_{i+j+2m-\mu}(-\bzeta)N_{i+j,m-\mu} \label{MZVrel_01}
\end{align}
\end{cor}

\subsection{Various examples of functional relations of the MPLs} \label{subsec:connection_0_1_regular:examples}

In what follows, computing the formula \eqref{MPLrel_01} for some lower $l$ and $m$, or by specializing the parameters, we show various concrete relations of the multiple polylogarithms of one variable.

\subsubsection{The case of $m=0$ and $l=1$}
By easy calculation we have
\begin{align}
&(\text{the left hand side of \eqref{MPLrel_01}})\big|_{m=0,l=1} \notag \\
&=-\Li_{k+1}(z)-\Li_{2,\underbrace{\scriptstyle 1,\ldots,1}_{k-1\text{times}}}(1-z) - \sum_{i=1}^k \Li_{i}(z)\Li_{\underbrace{\scriptstyle 1,\ldots,1}_{k-i+1\text{times}}}(1-z), \\
\intertext{on the other hand,}
&(\text{the right hand side of \eqref{MPLrel_01}})\big|_{m=0,l=1} \notag \\
&=\sum_{i=0}^k \sum_{j=0}^1 \binom{i+j}{i}P_{k-i}(\bzeta)P_{1-j}(\bzeta)P_{i+j}(-\bzeta) =\sum_{i=0}^k (i+1)P_{k-i}(\bzeta)P_{i+1}(-\bzeta) \notag\\
&=\text{the coefficient of $z^k$ in } \exp(\sum_{n=2}^{\infty} \frac{\zeta(n)}{n}z^n) \frac{d}{dz}\exp(-\sum_{n=2}^{\infty} \frac{\zeta(n)}{n}z^n) \notag \\
&=\text{the coefficient of $z^k$ in } -\sum_{i=2}^{\infty}\zeta(i)z^{i-1} \notag \\
&= -\zeta(k+1).
\end{align}
Consequently we obtain
\begin{equation}
\Li_{k+1}(z)+\Li_{2,\underbrace{\scriptstyle 1,\ldots,1}_{k-1\text{times}}}(1-z) + \sum_{i=1}^k \Li_{i}(z)\Li_{\underbrace{\scriptstyle 1,\ldots,1}_{k-i+1\text{times}}}(1-z)=\zeta(k+1).
\end{equation}
This is known as Euler's inversion formula for polylogarithms.

\subsubsection{The case of $m=0$ and $l=2$}

Similarly we have
\begin{align}
&(\text{the left hand side of \eqref{MPLrel_01}})\big|_{m=0,l=2} \notag \\
&=-\Li_{k+1,1}(z)-\Li_{3,\underbrace{\scriptstyle 1,\ldots,1}_{k-1 \text{times}}}(1-z) \notag \\
&\qquad \qquad -\Li_{1}(z)\Li_{2,\underbrace{\scriptstyle 1,\ldots,1}_{k-1 \text{times}}}(1-z) - \sum_{i=1}^k \Li_{i,1}(z)\Li_{\underbrace{\scriptstyle 1,\ldots,1}_{k-i+1\text{times}}}(1-z), \\
\intertext{and}
&(\text{the right hand side of \eqref{MPLrel_01}})\big|_{m=0,l=2} \notag \\
&=\sum_{i=0}^k \sum_{j=0}^2 \binom{i+j}{i}P_{k-i}(\bzeta)P_{l-j}(\bzeta)P_{i+j}(-\bzeta) \notag \\
&=\sum_{i=0}^k \frac{(i+2)(i+1)}{2}P_{k-i}(\bzeta)P_{i+2}(-\bzeta) \notag \\
&=\text{the coefficient of $z^k$ in }\frac{1}{2}\exp(\sum_{n=2}^{\infty} \frac{\zeta(n)}{n}z^n) \frac{d^2}{dz^2}\exp(-\sum_{n=2}^{\infty} \frac{\zeta(n)}{n}z^n) \notag \\
&=-\frac{k+1}{2}\zeta(k+2)+\frac{1}{2}\sum_{i=1}^{k-1}\zeta(i+1)\zeta(k-i+1).
\end{align}

Therefore, we have
\begin{align}
&\Li_{k+1,1}(z)+\Li_{3,\underbrace{\scriptstyle 1,\ldots,1}_{k-1 \text{times}}}(1-z)+\Li_{1}(z)\Li_{2,\underbrace{\scriptstyle 1,\ldots,1}_{k-1 \text{times}}}(1-z) \notag \\
&\qquad + \sum_{i=1}^k \Li_{i,1}(z)\Li_{\underbrace{\scriptstyle 1,\ldots,1}_{k-i+1\text{times}}}(1-z) \notag \\
&=\frac{k+1}{2}\zeta(k+2)-\frac{1}{2}\sum_{i=1}^{k-1}\zeta(i+1)\zeta(k-i+1).
\end{align}

Especially taking the limit as $z$ tends to $1$, we get the formula shown by Euler(\cite{Z1});
\begin{equation}
\zeta(k+1,1)-\frac{k+1}{2}\zeta(k+2)+\frac{1}{2}\sum_{i=1}^{k-1}\zeta(i+1)\zeta(k-i+1)=0. \label{Euler'sZetaRelation}
\end{equation}

\subsubsection{The sum formula for the MPLs}

Regarding the coefficients of $\alpha^1$ in both sides of \eqref{connection01_11} as a series in $p=(1-\gamma), q'=(\beta+1-\gamma)$, we obtain the following proposition;

\begin{prop} \label{prop:MPL_sum_formula}
For any positive integers $k>n>0$,
\begin{align}
\sum_{s}&G_0(k,n,s;z)+\sum_{s}G_0(k,k-n,s;1-z) \notag \\
&+\sum_{\substack{k'+k''=k\\n'+n''=n}}\sum_{s'}G_1(k',n',s';z)\sum_{s''}G_1(k'',k''-n'',s'';1-z) = \zeta(k). \label{MPL_sum_formula}
\end{align}
\end{prop}

Particularly, the limit of the formula \eqref{MPL_sum_formula} as $z \to 1$ is the sum formula with the multiple zeta values shown by Granville(\cite{Gr})-Zagier(\cite{Z2});
\begin{equation}
\sum_{s}G_0(k,n,s;1) = \zeta(k).
\end{equation}

\begin{proof}[Proof of Proposition \ref{prop:MPL_sum_formula}]

\begin{equation}
\frac{d}{d\alpha}\text{(the right hand side of \eqref{connection01_11})} \Big|_{\alpha \to 0} =-\sum_k \zeta(k+1) p^k + \sum_l \zeta(l+1) (\beta+1-\gamma)^l
\end{equation}
On the other hand, 
\begin{align}
&\text{the coefficient of $p^k (\beta+1-\gamma)^l$ in the left hand side of \eqref{connection01_11}} \notag \\
&=\sum_{m} G_0(k+l+1,l,m;z) - \sum_{m} G_0(k+l+1,l+1,m;z) \notag \\
&+ \sum_{m} G_0(k+l+1,k+1,m;1-z) - \sum_{m} G_0(k+l+1,k,m;1-z) \notag \\
& \qquad + \sum_{\substack{k'+k''=k+1\\l'+l''=l}} \sum_{m'}G_1(k'+l'-1,l',m';z)G_1(k''+l''+1,k''+1,m'';1-z) \notag \\
& \qquad - \sum_{\substack{k'+k''=k\\l'+l''=l+1}} \sum_{m'}G_1(k'+l'-1,l',m';z)G_1(k''+l''+1,k''+1,m'';1-z).
\end{align}
Therefore, we obtain
\begin{align}
&\sum_{s}G_0(k,n,s;z)+\sum_{s}G_0(k,k-n,s;1-z) \notag \\
&\qquad \qquad +\!\!\sum_{\substack{k'+k''=k\\n'+n''=n}}\!\!\sum_{s'}G_1(k',n',s';z)\sum_{s''}G_1(k'',k''-n'',s'';1-z) \notag \\
&=\sum_{s}G_0(k,n-1,s;z)+\sum_{s}G_0(k,k-(n-1),s;1-z)\notag\\
&\qquad +\!\!\!\!\sum_{\substack{k'+k''=k\\n'+n''=n-1}}\!\!\sum_{s'}G_1(k',n',s';z)\sum_{s''}G_1(k'',k''-n'',s'';1-z) \notag \\*
&\hspace{8.5cm} (k>n, n>1) \\
\intertext{and}
&\sum_{s}G_0(k,1,s;z)+\sum_{s}G_0(k,k-1,s;1-z)\notag\\
&\qquad +\sum_{\substack{k'+k''=k\\n'+n''=1}}\sum_{s'}G_1(k',n',s';z)\sum_{s''}G_1(k'',k''-n'',s'';1-z)=\zeta(k).
\end{align}
We have thus proved the proposition.
\end{proof}

\section{Functional relations derived from the connection relations between irregular solutions}

In this section, we consider functional relations of the multiple polylogarithms of one variable derived from the connection relations between irregular solutions to the hypergeometric equation. The construction of solutions shown in \S \ref{sec:show_HGF_by_MPL} and \S \ref{sec:connection_0_1_regular} can be applied whether the solution is regular or not. Then we can obtain the functional relations with respect to the $(1,2)$, $(2,1)$ and $(2,2)$-element of the connection relation \eqref{connection01} between $z=0$ and $z=1$ and the connection relation \eqref{connection0infty} between $z=0$ and $z=\infty$ in a similar way as above. But, in general, these relations are too complicated to be described explicitly.

In what follows, we give the functional relations derived from the limit of the $(1,1)$-element of \eqref{connection0infty} as $\beta$ tends to $0$. The results include Euler's inversion formula between $z=0$ and $\infty$ and the zeta values at even positive integers in the limit as $z \to 1$.

\subsection{The fundamental solution matrix in the neighborhood of $z=\infty$}

Let $u=\frac{1}{z}$ be a complex coordinate of $z=\infty$ and $\rho_{\infty}: \fH' \to M(2,\C)$ be a representation defined by
\begin{equation}
\rho_{\infty}(X)=\rho_0(Y)-\rho_0(X)=\begin{pmatrix}0&\beta\\\alpha&\alpha+\beta\end{pmatrix},\;\; \rho_{\infty}(Y)=\rho_0(Y)=\begin{pmatrix}0&0\\\alpha&q\end{pmatrix}. \label{rhoinftydef}
\end{equation}
One can rewrite the hypergeometric equation \eqref{HGeq3} as
\begin{equation}
\frac{d}{du}G=\left(\frac{\rho_{\infty}(X)}{u}+\frac{\rho_{\infty}(Y)}{1-u}\right)G, \label{HGeq_infty}
\end{equation}
and get the fundamental solution matrix $\rho_{\infty}(H_0(u))$. Comparing the asymptotic properties, we have
\begin{equation}
\Phi_{\infty}=\rho_{\infty}(H_0(u)) \begin{pmatrix}1&1\\-\frac{\alpha}{\beta}&-1\end{pmatrix}.
\end{equation}
Then we obtain
\begin{align}
&\Phi_{\infty}^{-1}=\frac{\beta}{\beta-\alpha}\begin{pmatrix}1&1\\-\frac{\alpha}{\beta}&-1\end{pmatrix}\rho_{\infty}(H_0(u))^{-1}, \notag \\
&\rho_{\infty}(H_0(u))^{-1}=\sum_{w}\Li(S(w);u)\rho_{\infty}(W). \label{sol_infty}
\end{align}

In \eqref{sol_infty}, it is difficult to calculate $\rho_{\infty}(W)$ concretely, but one can see the limit of $\rho_{\infty}(W)$ and $\rho_{\infty}(H_0(u))^{-1}$ as $\beta \to 0$ as follows.

\begin{lem}
For any word $W$ in $\fH$, we have
\begin{equation}
\lim_{\beta \to 0} \rho_{\infty}(W) = \begin{cases}\alpha^{|W|-d(w)}(\alpha+p)^{d(W)-1} \begin{pmatrix}0&0\\\alpha&\alpha+p\end{pmatrix} & (W \in \fH Y)\\ \alpha^{|W|-d(w)}(\alpha+p)^{d(W)-1} \begin{pmatrix}0&0\\\alpha+p&\alpha+p\end{pmatrix} & (W \in \fH X).\end{cases}
\end{equation}
\end{lem}

\begin{prop}
The following formula holds:
\begin{equation}
\lim_{\beta \to 0} \rho_{\infty}(H_0(u))^{-1}=\begin{pmatrix}1&0\\H_{21}& H_{22}\end{pmatrix},
\end{equation}
where
\begin{align}
&H_{21}= \sum_{k=1}^{\infty}(-1)^k \frac{\log^k u}{k!} \alpha^k +\sum_{k\ge n \ge 1} \Li_{\underbrace{\scriptstyle 1,\ldots,1}_{n \text{times}}}(u)\frac{\log^{k-n}u}{(k-n)!}\alpha^{k-n}(\alpha+p)^n \notag \\
&\quad +p \sum_{k>n\ge 1}(-1)^k\!\!\sum_{i=0}^{k-n-1}(-1)^{k-n-1-i} \!\!\Li_{k-n-i+1,\underbrace{\scriptstyle 1,\ldots,1}_{n-1 \text{times}}}(u) \frac{\log^i(u)}{i!}\alpha^{k-n}(\alpha+p)^n, \\
&H_{22}=\displaystyle \sum_{k\ge n \ge 0}(-1)^k \Li_{\underbrace{\scriptstyle 1,\ldots,1}_{n \text{times}}}(u)\frac{\log^{k-n}u}{(k-n)!}\alpha^{k-n}(\alpha+p)^n.
\end{align}
\end{prop}

\subsection{The functional relations derived from the $(1,1)$-element of the connection relation between $z=0$ and $\infty$}

Multiplying both sides of the $(1,1)$-elements of the connection relation \eqref{connection0infty}
\begin{equation}
(\Phi_{\infty}^{-1})_{11}(\Phi_0)_{11} + (\Phi_{\infty}^{-1})_{12}(\Phi_0)_{21} = (C^{0\infty})_{11} \label{connection0infty_11}
\end{equation}
by $\frac{\beta-\alpha}{\beta}$ and then taking the limit as $\beta \to 0$, we obtain the following relations as the coefficients of $\alpha^m (\alpha+1-\gamma)^n$.

\begin{prop} \label{MPL0infty}
For any positive integers $m,n$, and $z \in \fU$, we have
\begin{align}
&\frac{\log^m \frac{1}{z}}{m!}- \sum_{i=0}^{m-1}\left(\Li_{m-i}(z)+(-1)^{m-i}\Li_{m-i}(\frac{1}{z})\right)\frac{\log^i \frac{1}{z}}{i!} = (-1)^m B_m \frac{(2 \pi i)^m}{m!}, \label{rel0infty_1}\\
&\sum^{m-1}_{i=0}(-1)^{n+i+1}\Li_{m-i,\underbrace{\scriptstyle 1,\ldots,1}_{n \text{times}}}(\frac{1}{z})\frac{\log^i \frac{1}{z}}{i!} +\sum_{i=0}^{m-1}(-1)^{n+i+1}\Li_{m-i+1,\underbrace{\scriptstyle 1,\ldots,1}_{n-1 \text{times}}}(\frac{1}{z}) \frac{\log^i \frac{1}{z}}{i!}\notag \\*
&\qquad +\sum_{i=0}^{m-1}\sum_{j=0}^n (-1)^{m+n-j-1}\Li_{\underbrace{\scriptstyle 1,\ldots,1}_{n-j \text{times}}}(\frac{1}{z})\sum_{k=0}^j \binom{m-i-1+j-k}{m-i-1}\notag \\*
&\hspace{6cm} \times \sum_sG_1(m-i+j,k+1,s;z)\frac{\log^{i}\frac{1}{z}}{i!} \notag \\
&\qquad = \sum_{\substack{m_1+m_2+m_3=m\\n_1+n_2=n}}\binom{m_1+n_1}{m_1}(-1)^{m_1}P_{m_1+n_1}(\bzeta)P_{m_2}(\bzeta)\frac{(-\pi i)^{m_3}}{m_3!}P_{n_2}(-\bzeta), \label{rel0infty_2}
\end{align}
where $B_m$ are the Bernoulli numbers, namely the real numbers introduced through the generating function $\sum_m B_m \frac{t^m}{m!}=\frac{te^t}{e^t-1}$.
\end{prop}
Especially, if $m$ is an even positive integer, we have
\begin{equation}
-2 \zeta(m)=B_m \frac{(2 \pi i)^m}{m!} \label{ZVeven}
\end{equation}
as the limit as $z$ tends to $1$ in \eqref{rel0infty_1}. Furthermore, taking the limit in \eqref{rel0infty_2} in the case of $n=1$ and $2$, we obtain the following proposition.

\begin{cor} \label{MZV0infty}
The following relations among the MZVs holds.
\begin{align}
(m+2)\zeta(m+1)&=2 \sum_{\substack{i+2k=m,\\i,k\ge 1}}\zeta(i+1)\zeta(2k) \hspace{3cm} (m:\text{odd}), \label{MZV0infty_n1odd}\\
2\zeta(m,1)&=m\zeta(m+1)-2\sum_{\substack{i+2k=m,\\i,k\ge 1}}\zeta(i+1)\zeta(2k)\hspace{1cm} (m:\text{even}), \label{MZV0infty_n1even}\\
(m+2)\zeta(m+1)&=-\frac{\pi^2}{2}\zeta(m)+\frac{(m+2)(m+1)}{2}\zeta(m+2) \notag \\
&\qquad -\sum_{\substack{i+2k=m,\\i,k\ge 1}}(i+1)\zeta(i+2)\zeta(2k) \notag \\
&\qquad -\sum_{\substack{i+j+2k=m,\\i,j,k\ge 1}}(i+1)\zeta(i+1)\zeta(j+1)\zeta(2k) \quad  (m:\text{odd}), \label{MZV0infty_n2odd}\\
2\zeta(m,1,1)&=-m\zeta(m+1,1)+\frac{\pi^2}{2}\zeta(m)-\frac{(m+2)(m+1)}{2}\zeta(m+2) \notag \\
&\qquad +\sum_{\substack{i+2k=m,\\i,k\ge 1}}(i+1)\zeta(i+2)\zeta(2k) \notag \\
&\qquad +\sum_{\substack{i+j+2k=m,\\i,j,k\ge 1}}(i+1)\zeta(i+1)\zeta(j+1)\zeta(2k) \quad (m:\text{even}). \label{MZV0infty_n2even}
\end{align}
\end{cor}

\begin{proof}[Proof of Proposition \ref{MZV0infty}]
From the definition of MPLs and choice of the branches of them on the interval $\fI$ defined in \S \ref{sec:Pre:Anal_MPL}, we have the following analytic continuation;
\begin{align}
&\Li_{\underbrace{\scriptstyle 1,\ldots,1}_{n \text{times}}}(\frac{1}{z})=\frac{1}{n!}(\Li_1(z)+\log z+\pi i)^n, \qquad (z \in \fI)
\intertext{and}
&\lim_{z\to 1 \text{ on} \fI}\Li_{k_1,k_2,\ldots,k_n}(\frac{1}{z})=\lim_{z\to 1 \text{ on} \fI}\Li_{k_1,k_2,\ldots,k_n}(z)=\zeta(k_1,k_2,\ldots,k_n). \qquad (k_1 \ge 2)
\end{align}

Here we note that $\Li(w,z)\log z \to 0$ as $z \to 1$ on $\fI$ for any word $w$ in $\fh$. By virtue of the analytic continuation above and the algebraic homomorphism $\Li(w_1;z)\Li(w_2;z)=\Li(w_1 \sh w_2;z)$, we obtain
\begin{align*}
&(\text{the left hand side of \eqref{rel0infty_2} at $n=1$}) \\*
&\quad \underset{z\to 1}\to (1+(-1)^m m)\zeta(m+1)+(1+(-1)^m)\zeta(m+1)+(-1)^m\zeta(m)\pi i,\\
&(\text{the left hand side of \eqref{rel0infty_2} at $n=2$}) \\*
&\quad \underset{z\to 1}\to ((-1)^{m+1}-1)\zeta(m,1,1)+((-1)^m m-1)\zeta(m+1,1)\\*
&\quad\qquad +(-1)^m\frac{\pi^2}{2}\zeta(m)+(-1)^{m+1}\frac{m(m+1)}{2}\zeta(m+2)\\*
&\quad\qquad +(-1)^{m+1}\zeta(m,1)\pi i+(-1)^m m \zeta(m+1)\pi i.
\intertext{On the other hand, differentiating the $(1,1)$-element of $C^{0\infty}$ by $(\alpha+1-\gamma)$, taking the limit as $(\alpha+1-\gamma) \to 0$ and applying the equation \eqref{ZVeven}, we have}
&(\text{the right hand side of \eqref{rel0infty_2} in $n=1$}) \\*
&\quad= (-1)^m\zeta(m+1) -2\sum_{\substack{i+2k=m\\i,k\ge 1}}(-1)^i \zeta(i+1)\zeta(2k) +(-1)^m \zeta(m) \pi i ,\\
&(\text{the right hand side of \eqref{rel0infty_2} in $n=2$}) \\*
&\quad= (-1)^m (m+1)\zeta(m+2)+(-1)^{m+1} \zeta(m+1,1)\\*
&\quad\qquad -\sum_{\substack{i+2k=m\\i,k\ge 1}}(-1)^i (i+1)\zeta(i+2)\zeta(2k)\\*
&\quad\qquad -\sum_{\substack{i+j+2k=m\\i,j,k\ge 1}}(-1)^{i+j} \zeta(i+1)\zeta(j+1)\zeta(2k)\\*
&\quad\qquad +(-1)^m \frac{1}{2}\Big(\sum_{i=1}^{m-2}\zeta(i+1)\zeta(m-i)+m\zeta(m+1)\Big)\pi i.
\end{align*}
The real parts of these results yield the equations from \eqref{MZV0infty_n1odd} to \eqref{MZV0infty_n2even}. The imaginary parts of them are nothing but the identity for $n=1$ and for $n=2$, the relation \eqref{Euler'sZetaRelation} previously shown.

\end{proof}

According to \cite{BBBL}, the generating function of the MZVs of $\zeta(m,1,\ldots,1)$ type can be expressed as a ratio of gamma functions. In this context, Proposition \ref{MZV0infty} partially gives the concrete expressions, and Proposition \ref{MPL0infty} is regarded as a generalization of this claim in \cite{BBBL} to the MPLs.

\section{Acknowledgments}

The author express his deep gratitude to Professor Yasuo Ohno for many valuable comments. He spoke about the contents of this article at Professor Yumiko Hironaka's seminar, Waseda University, May 16, 2008. He also thanks Professor Kimio Ueno, his academic supervisor, Dr. Jun-ichi Okuda and members of Ueno's laboratory for useful advice and discussions.

{\noindent
Shu Oi\\
Major in Mathematical Science\\
Graduate School of Science and Engineering\\
Waseda University\\
3-4-1, Okubo Shinjuku-ku\\
Tokyo 169-8555, Japan\\\\
{\tt shu\_oi@toki.waseda.jp}\\
}

\end{document}